\documentclass{daj}

\dajAUTHORdetails{%
  title = {A polynomial Freiman-Ruzsa inverse theorem for function fields},
  author = {Thomas F. Bloom},
  plaintextauthor = {Thomas F. Bloom},
}   %%% END \dajAUTHORdetails

\dajEDITORdetails{%
   year={2025},
   number={24},
   received={29 January 2025},   % received date: example: 7 January 2017
   published={8 October 2025},  % published date
   doi={10.19086/da.145176},       % XXX = number of paper, e.g. da006 for paper#6
}   %%% END \dajEDITORdetails

\usepackage{amsmath,amsfonts,amsthm,amssymb,stmaryrd}

\renewcommand{\leq}{\leqslant}
\renewcommand{\geq}{\geqslant}

\newcommand{\bbf}{\mathbb{F}}

\newcommand{\bbz}{\mathbb{Z}}
\newcommand{\bbr}{\mathbb{R}}
\newcommand{\bbq}{\mathbb{Q}}

\newcommand{\fk}{\mathbb{F}_p(\!(t^{-1})\!)}

\newcommand{\Pol}[1]{[#1]}
\newcommand*{\bbe}{
  \mathop{
    \mathchoice{\vcenter{\hbox{\larger[4]$\mathbb{E}$}}}
               {\kern0pt\mathbb{E}}
               {\kern0pt\mathbb{E}}
               {\kern0pt\mathbb{E}}
  }\displaylimits
}

\newcommand{\abs}[1]{\left\lvert #1\right\rvert}
\newcommand{\Abs}[1]{\lvert #1\rvert}

\newtheorem{theorem}{Theorem}
\newtheorem{lemma}{Lemma}

\newtheorem{corollary}{Corollary}
\newtheorem{conjecture}{Conjecture}

\newtheorem{question}{Question}
\theoremstyle{definition}
\newtheorem{definition}{Definition}
\newtheorem{remark}{Remark}
\begin{document}

\begin{frontmatter}[classification=text]
%% EDITOR: this will force the keywords to appear right after the Abstract.
%%   If the abstract is too long and would force the keywords off the
%%   front page, please comment out % [classification=text] above
%%   This way the keywords will be floated on the bottom of the first page
%%   even though the Abstract spills over to the next page.

%%% AUTHOR: Title goes here.  This line is optional.  You must use it
%%   if title has footnote attached or requires nontrivial typesetting,
%%   e.g., inclusion of linebreaks to force nice layout.

\title{A polynomial Freiman-Ruzsa inverse theorem for function fields}

%%% AUTHOR:
%%% List all authors. If you wish, place grant acknowledgements in \thanks.
%%% In brackets include a short tag for each author.
\author[tb]{Thomas F. Bloom\thanks{Supported by a Royal Society University Research Fellowship.}}

\begin{abstract}Using the recent proof of the polynomial Freiman-Ruzsa conjecture over $\bbf_p^n$ by Gowers, Green, Manners, and Tao, we prove a version of the polynomial Freiman-Ruzsa conjecture over function fields. In particular, we prove that if $A\subset\mathbb{F}_p[t]$ satisfies $\abs{A+tA}\leq K\abs{A}$ then $A$ is efficiently covered by at most $K^{O(1)}$ translates of a generalised arithmetic progression of rank $O(\log K)$ and size at most $K^{O(1)}\abs{A}$. 

As an application we give an optimal lower bound for the size of $A+\xi A$ where $A\subset\fk$ is a finite set and $\xi\in \fk$ is transcendental over $\bbf_p[t]$.
\end{abstract}
\end{frontmatter}
%\address{Department of Mathematics, University of Manchester, Manchester, M13 9PL}
%\email{thomas.bloom@manchester.ac.uk}
\maketitle

The goal of inverse sumset theorems in additive combinatorics is to deduce strong structural properties of finite sets $A$ from weak combinatorial assumptions such as small doubling\footnote{We use the Vinogradov notation $X\ll Y$ to mean there is some unspecified constant $C>0$ such that $X\leq CY$.} $\abs{A+A}\ll \abs{A}$. Usually such questions are considered in either finite abelian groups (such as $\mathbb{Z}/N\bbz$ or $\bbf_p^n$) or $\bbz$. The latter is particularly important for applications to number theory. 

In this note we will instead work in the `function field' setting of $\bbf_q[t]$, where $\bbf_q$ is a fixed\footnote{In particular all implicit constants in this paper may depend on both the size of the field and its characteristic.} finite field of size $q=p^r$ for some prime $p$, and $t$ is an indeterminate. There is a rich tradition of studying the analogue of classic number theoretic questions in $\bbf_q[t]$, and often the rigid algebraic structure of the latter allows for stronger results.

We will show that this is also the case for inverse sumset theorems, and in particular the recent breakthrough in our quantitative understanding of inverse sumset theorems in $\bbf_p^n$ by Gowers, Green, Manners, and Tao \cite{GGMT} leads to similarly strong quantitative inverse results in $\bbf_q[t]$. Analogously strong quantitative results for $\bbz$ remain, as yet, out of reach.

We must be careful how we ask the question: if we only consider a finite set $A\subset \bbf_q[t]$ and its sumset $A+A$ then we only see the additive aspect of the ring, and so this is trivially equivalent to studying the inverse sumset problem in $\bbf_p^n$ for some $n$, which is the result of \cite{GGMT}. This is not, however, the correct inverse hypothesis for $\bbf_q[t]$, since it does not `see' the interaction of $A$ with the indeterminate $t$. To obtain a more interesting and useful function field analogue we need to track the interaction of $A$ with both addition and multiplication by $t$. The most obvious way of doing this is to bound the size, not of $A+A$, but of
\[A+tA = \{ a+bt : a,b\in A\}.\]

This leads to the following question.
\begin{question}
If $A\subset \bbf_q[t]$ is a finite set with $\abs{A+tA}\ll \abs{A}$ then what can be deduced about the structure of $A$?
\end{question}
We remark that $\Abs{A+tA}\ll \abs{A}$ immediately implies $\abs{A+A}\ll \abs{A}$ via the Ruzsa triangle inequality, but it is a strictly stronger assumption. For example, if we take $A$ to be the $\bbf_q$-linear span of $1,t^2,t^4,\ldots,t^{2n}$ then $\abs{A+A}=\abs{A}$ yet $\abs{A+tA}=\abs{A}^2$.

An obvious candidate for a set $A$ such that $\abs{A+tA}\ll \abs{A}$ is the initial segment $A=\{ x: \deg x <n\}$, where we have $\abs{A+tA}=q\abs{A}$. This is analogous to the observation that, in $\bbz$, the interval $\{1,\ldots,N\}$ is a natural example of a set with small doubling. Such examples can be generalised in the usual fashion, taking translates, dilates, sumsets, and passing to large subsets. 

The philosophy of inverse sumset results is that these natural examples are the only examples. In $\bbz$ this leads to the conclusion that $A$ is covered by a few translates of a low-dimensional arithmetic progression. The conclusion in $\bbf_q[t]$ is very similar, once we have defined what an arithmetic progression is in the natural way.

\begin{definition}[Generalised arithmetic progressions]
A generalised $\bbf_q[t]$-progression of rank $\leq d$ is a set of the shape
\[x_0+\Pol{n_1}\cdot x_1+\cdots+\Pol{n_d}\cdot x_d\]
for some $n_1,\ldots,n_d\geq 1$ and $x_0,x_1,\ldots,x_d\in \bbf_q[t]$, where 
\[\Pol{n} = \{ x \in \mathbb{F}_q[t] : \deg x < n\}.\]
\end{definition}

We can now state precisely the main result of this note, which is a quantitatively optimal inverse result in $\bbf_q[t]$, and should be viewed as the analogue of the polynomial Freiman-Ruzsa conjecture over $\bbf_p^n$, which was recently proved by Gowers, Green, Manners, and Tao \cite{GGMT}. 

We first state the result in the case when $q=p$, which is a little simpler. (Note we write $\langle A\rangle$ for the finite $\bbf_q$-vector space spanned by $A$.)
\begin{theorem}\label{th-pfrff}
Let $K_1,K_2\geq 2$. If $A\subseteq \bbf_p[t]$ is a finite set and
\[\abs{A+A}\leq K_1\abs{A}\textrm{ and }\abs{A+tA}\leq K_2\abs{A}\]
then there is a generalised $\bbf_p[t]$-progression $P\subseteq \langle A\rangle$ of rank $O(\log K_2)$ and size 
\[K_2^{-O(1)}\abs{A}\leq \abs{P}\leq K_1^{O(1)}\abs{A}\]
such that 
\begin{enumerate}
\item $A\subseteq P+X$ for some  finite set $X\subseteq \bbf_p[t]$ of size $\abs{X}\leq K_2^{O(1)}$ and
\item $P\subseteq A-A+X'$ for some finite set $X'\subseteq P$ of size $\abs{X'}\leq K_1^{O(1)}$.
\end{enumerate}
\end{theorem}

In the above we have decoupled the two doubling constants, since this is useful in some applications. Since $\abs{A+A}\leq \abs{A+tA}^2/\abs{A}$ by the Pl\"{u}nnecke-Ruzsa inequality (see for example \cite[Corollary 6.28]{TV}), we can immediately deduce the following simpler form.
\begin{corollary}\label{cor-pfrff}
Let $K\geq 2$. If $A\subseteq \bbf_p[t]$ is a finite set and $\abs{A+tA}\leq K\abs{A}$ then there is a generalised $\bbf_p[t]$-progression $P$ of rank $O(\log K)$ and size $\abs{P}\leq K^{O(1)}\abs{A}$, together with a set $X\subseteq \bbf_p[t]$ of size $\abs{X}\leq K^{O(1)}$, such that $A\subseteq P+X$.
\end{corollary}
The strength of this result is perhaps slightly surprising given that the analogous naive formulation of the polynomial Freiman-Ruzsa conjecture is actually false over the integers, as shown by Lovett and Regev \cite{LR17} -- for quantitative bounds of polynomial strength one must work not just with generalised arithmetic progressions but with the more general concept of `convex progressions' (linear images of the intersection of a lattice with a symmetric convex body).

The case when $q=p$ is simpler because in controlling $A+tA$ we control not only $A+A$ but also $A+\lambda A$ for all $\lambda \in \bbf_p$. Indeed, if $\abs{A+tA}\leq K\abs{A}$ then, by the Pl\"{u}nnecke-Ruzsa inequality, $\abs{A+\cdots +A}\leq K^{n}\abs{A}$ for all integers $n\geq 1$ (where there are $n$ copies of the summand $A$), and hence $\abs{A+\lambda A}\leq K^{O_p(1)}\abs{A}$ for all $\lambda\in \bbf_p$.

For general fields there is no way to control $\abs{A+\lambda A}$ for all $\lambda \in \bbf_q$ by $\abs{A+tA}$ alone, and so for the general statement we need to add an additional hypothesis.

\begin{theorem}\label{th-gen}
Let $K_1,K_2\geq 2$. If $A\subseteq \bbf_q[t]$ is a finite set and
\[\abs{A+\lambda A}\leq K_1\abs{A}\textrm{ and }\abs{A+tA}\leq K_2\abs{A}\]
for all $\lambda \in \bbf_q$ then there is a generalised $\bbf_q[t]$-progression $P\subseteq \langle A\rangle$ of rank $O(\log K_1K_2)$ and size 
\[(K_1K_2)^{-O(1)}\abs{A}\leq \abs{P}\leq K_1^{O(1)}\abs{A}\]
such that 
\begin{enumerate}
\item $A\subseteq P+X$ for some  finite set $X\subseteq \bbf_q[t]$ of size $\abs{X}\leq (K_1K_2)^{O(1)}$ and
\item $P\subseteq A-A+X'$ for some finite set $X'\subseteq P$ of size $\abs{X'}\leq K_1^{O(1)}$.
\end{enumerate}
\end{theorem}

Theorem~\ref{th-pfrff} is an immediate consequence of Theorem~\ref{th-gen} and so we will prove only the latter, which is a consequence of the main result of \cite{GGMT}, together with the following structural result (the only novelty in this paper).

\begin{theorem}\label{th:vecspace}
If $V\subset \bbf_q[t]$ is a finite $\bbf_q$-vector space and 
\[\abs{V+tV}\leq q^k\abs{V}\]
then $V$ is a generalised $\bbf_q[t]$-progression of rank at most $k$.
\end{theorem}

\begin{remark}\label{rem}
Our proof of Theorem~\ref{th:vecspace} is self-contained and elementary, but Will Sawin has observed that it should also follow the classification of representations of Kronecker quivers (which in this context are simply pairs of vector spaces $V$ and $W$, with an associated pair of linear maps from $V$ to $W$).

Our representation is formed by taking $W=V+tV$ and the maps $v\mapsto v$ and $v\mapsto tv$. This representation can be written as the direct sum of irreducible components, the possibilities of which are given by the aforementioned classification. The fact that any non-trivial linear combination of the linear maps in question is injective means that most of the possibilities in this classification should be able to be ruled out, leaving the only possibilities corresponding to some $1$-dimensional progression $[n]x\leq V$. It follows that $V$ is a direct sum of spaces of the shape $[n]x$, and as such is a generalised progression, and comparing the dimensions of $V$ and $W$ implies the rank of this progression is at most $k$.

We have not found a clean statement of the relevant classification (which is usually only considered over algebraically closed fields, rather than finite fields) suitable to make this sketch precise. Since our focus is on additive combinatorics, and Theorem~\ref{th:vecspace} has an elementary proof, we end this digression here.
\end{remark}

We give our proof of Theorem~\ref{th:vecspace} in Section~\ref{sec-vecproof}, and now deduce Theorem~\ref{th-gen}. 
 
\begin{proof}[Proof of Theorem~\ref{th-gen}]
We will use the notation and definitions from \cite{GGMT}. In particular, we write $d[X;Y]$ for the entropic Ruzsa distance between random variables $X$ and $Y$ defined by
\[d[X;Y] = H(X'+Y')-(H(X')+H(Y'))/2\]
where $X',Y'$ are independent copies of $X$ and $Y$ respectively and $H$ is the Shannon entropy. We will use standard properties of entropic distance as recorded in \cite[Proposition A.1]{GGMT} without further mention. In particular, if $U_A$ is the uniform distribution on $A$ then $d[U_A;U_A]\ll \log K_1$. By \cite[Theorem 1.3]{GGMT} there exists some finite subgroup $H\subseteq \langle A\rangle$ such that $d[U_A;U_H]\ll \log K_1$. As explained in \cite[Appendix B]{GGMT} the fact that $d[U_A;U_H]\ll \log K_1$ implies that $\abs{H}\ll K_1^{O(1)}\abs{A}$ and $A$ is covered by $K_1^{O(1)}$ many translates of $H$.

Furthermore, since $d[U_A;U_{t A}]\leq \log K_2$ we further deduce via the Ruzsa triangle inequality (in entropic form) that $d[U_H;U_{t H}]\ll \log(K_1K_2)$. Similarly, since $d[U_A;U_{\lambda A}]\ll \log K_1$ for all $\lambda \in \bbf_q$ we have $d[U_H; U_{\lambda H}]\ll \log K_1$ also, and hence $d[U_H;U_{\lambda tH}]$ by the triangle inequality.

It follows (again arguing as in \cite[Appendix B]{GGMT}) that, for any $\lambda \in \bbq^*$, there is some $x_\lambda$ such that

\[\abs{(\lambda H)\cap (H+x_\lambda)}\gg K_1^{-O(1)}\abs{H}\]
and $x_\lambda'$ such that
\[\abs{(t \lambda H)\cap (H+x_\lambda')}\gg (K_1K_2)^{-O(1)}\abs{H}.\]
Since both $\lambda H$ and $t\lambda H$, for all $\lambda \in\bbf_q^*$, are all finite subgroups of the same size, this implies that
\[\abs{H'}\gg (K_1K_2)^{-O(1)}\abs{H},\]
where 
\[H'= \bigcap_{\lambda \in \bbf_q}(\lambda H)\cap \bigcap_{\lambda \in \bbf_q}(\lambda tH).\]
Note that $H'$ is a finite $\bbf_q$-vector space. Since $H'\leq H$ we can cover $H$ by $ (K_1K_2)^{O(1)}$ many translates of $H'$, hence $A$ is also covered by $(K_1K_2)^{O(1)}$ many translates of $H'$. Finally, we note that $H'+t H'\subseteq t H$, and hence 
\[\abs{H'+t H'}\leq \abs{H}\leq (K_1K_2)^{O(1)}\abs{H'}.\]
Theorem~\ref{th:vecspace} implies that $H'$ is a progression of rank $O(\log K_1K_2)$. 

For the final part we observe that by Ruzsa's covering lemma \cite[Lemma 2.14]{TV}) there exists some $X'$ of size
\[\abs{X'}\leq \frac{\abs{H+A}}{\abs{A}}\leq K_1^{O(1)}\]
such that $H'\subseteq H\subseteq A-A+X'$.
\end{proof}

Finally, we remark that Theorem~\ref{th:vecspace} was proved some years ago as part of the author's PhD thesis but never published. The spectacular advances of Gowers, Green, Manners, and Tao \cite{GGMT} mean that the consequences are now quite strong. Since those interested in applications to number theory in function fields would find these consequences interesting and perhaps useful, we thought it was time that they were published.

In Section~\ref{sec-vecproof} we prove Theorem~\ref{th:vecspace}. In Section~\ref{sec-dil} we give an application of this inverse result to the study of sums of transcendental dilates.
\section{Proof of Theorem~\ref{th:vecspace}}\label{sec-vecproof}

If $V$ is a finite $\bbf_q$-vector space then so is $V+t V$, and hence $\abs{V+t V}/\abs{V}=q^r$ for some $r\geq 0$. With this in mind, we call $r=\log_q(\abs{V+t V}/\abs{V})$ the \emph{weak arithmetic dimension} of $V$. 

An easy example of a $V$ with small weak arithmetic dimension is any space of the form
\begin{equation}\label{eq-pol}
V=\Pol{d_1}\cdot x_1+ \cdots +\Pol{d_k}\cdot x_k
\end{equation}
for some $d_1,\dots,d_k\geq 1$ and $x_1,\ldots,x_k\in \bbf_q[t]$. Such a $V$ has weak arithmetic dimension at most $k$. With this in in mind, we define the \emph{structural arithmetic dimension} of $V$ to be the smallest $k$ such that \eqref{eq-pol} holds. Observe in particular that, if we consider $d_1=\cdots=d_k=1$, the structural arithmetic dimension of $V$ is at most the dimension of $V$ considered as a vector space over $\bbf_q$. Summarising this discussion, we have the chain of inequalities 
\[\dim_{\mathrm{weak}}(V)\leq \dim_{\mathrm{struc}}(V)\leq \dim(V).\]
It is less obvious that in fact we always have equality $\dim_{\mathrm{weak}}(V)= \dim_{\mathrm{struc}}(V)$. This is the content of Theorem~\ref{th:vecspace}. 

For the proof of this we will use the following technical lemma. This is essentially the one variable case of the fact that a minimal Gr\"{o}bner basis always exists for any finite set in a polynomial ring; in this special case the proof is particularly simple.
\begin{lemma}\label{l-grobner}
If $V\subset \bbf_q[t]$ is a finite $\bbf_q$-vector space then there exists a decomposition of the form
\[V=\bbf_q\cdot x_1\oplus\cdots\oplus \bbf_q\cdot x_d,\]
where $\deg x_1<\deg x_2<\cdots<\deg x_d$.
\end{lemma}
\begin{proof}
We use induction on $\dim V$. If $\dim V=0$ then the result is trivial. Otherwise, let $x\in V\backslash\{0\}$ be a monic polynomial of minimal degree, and let $W$ be such that $V=\bbf_q\cdot x\oplus W$. The result follows from the inductive hypothesis and the fact that if $w\in W\backslash\{0\}$ then $\deg w>\deg x$. This latter fact is true because otherwise we must have a monic $w\in W\backslash \{0\}$ such that $\deg w=\deg x$, and hence $w-x\in V\backslash\{0\}$ has degree strictly less than $\deg x$, which contradicts the minimality of $\deg x$. 
\end{proof}

We will now prove Theorem~\ref{th:vecspace}. The proof is inductive, and the trick is to choose the right inductive hypothesis, which appears stronger than we strictly need.

\begin{proof}[Proof of Theorem~\ref{th:vecspace}]
For the purposes of this proof, we introduce yet another notion of dimension: $V$ has strong structural arithmetic dimension $k=\dim_{\mathrm{strong}}(V)$ if $k$ is minimal such that there exist $d_1,\ldots,d_k\geq 1$ and $x_1,\ldots,x_k\in\bbf_q[t]$ such that
\[V=\Pol{d_1}\cdot x_1\oplus \cdots\oplus \Pol{d_k}\cdot x_k\]
and $d_1+\deg x_1<\cdots<d_k+\deg x_k$. Note by Lemma~\ref{l-grobner} we may always take $d_1=\cdots=d_k=1$ and $k=\dim(V)$, and so $\dim_{\mathrm{strong}}(V)$ exists and is at most $\dim(V)$. In general, the following chain of inequalities is immediate from the definitions:
\[\dim_{\mathrm{weak}}(V)\leq \dim_{\mathrm{struc}}(V)\leq \dim_{\mathrm{strong}}(V)\leq \dim(V).\]
We will show by induction on $\dim_{\mathrm{strong}}$ that $\dim_{\mathrm{strong}}(V)\leq \dim_{\mathrm{weak}}(V)$, and hence certainly we have proved that always $\dim_{\mathrm{weak}}=\dim_{\mathrm{struc}}=\dim_{\mathrm{strong}}$. In other words, we will show for all $r\geq 0$ that if $\abs{V+t V}= q^r\abs{V}$ then $\dim_{\mathrm{strong}}(V)\leq r$. 

The case $r=0$ is trivial, since if $x\in V$ is an element of maximal degree then $V\cup \{tx\}\subseteq V+t V$, and if $tx\in V$ then $x=0$. In particular, if $\abs{V+t V}\leq \abs{V}$ then $V=\{0\}$ as required. We shall hence assume that $r\geq 1$ and that the claim has been proved for $r'<r$.

Suppose that $\abs{V+t V}=q^r\abs{V}$. Let 
\[V=\bbf_q\cdot x_1\oplus\cdots\oplus \bbf_q\cdot x_\ell\]
with $\deg x_1<\cdots <\deg x_\ell$, as provided by Lemma~\ref{l-grobner}, and for $1\leq s\leq \ell$ let
\[V_{\leq s}=\bbf_q\cdot x_1\oplus\cdots\oplus \bbf_q\cdot x_s.\]
We observe that if $x\in V_{\leq s}\backslash \{0\}$ then $\deg x_1\leq \deg x\leq \deg x_s$. Furthermore, if $x\in V\backslash V_{\leq s}$ then $\deg x>\deg y$ for all $y\in V_{\leq s}$. Let $1\leq s\leq \ell$ be maximal such that $V_{\leq s}$ has strong arithmetic dimension at most $r$. If $s=\ell$ then we are done, having shown that the strong arithmetic dimension of $V$ is at most $r$ as required. Suppose then, otherwise, that $1\leq s<\ell$ and that $V_{\leq s}$ has strong arithmetic dimension $1\leq r'\leq r$. We must have $r'=r$, or this contradicts the maximality of $s$ by considering the decomposition
\[V_{\leq s+1}=V_{\leq s}\oplus \bbf_q\cdot x_{s+1}=\Pol{d_1}\cdot y_1\oplus\cdots \oplus \Pol{d_{r'}}\cdot y_{r'}\oplus \bbf_q\cdot x_{s+1}.\]
Hence $V_{\leq s}$ has strong arithmetic dimension $r$, whence we have some decomposition 
\[V_{\leq s}=\Pol{d_1}\cdot y_1\oplus\cdots\oplus \Pol{d_r}\cdot y_r,\]
such that $d_1+\deg y_1<\cdots < d_r+\deg y_r$; furthermore, since $\deg x_s$ is maximal over all $x\in V_{\leq s}$, and so is $t^{d_r-1}y_r$, we have that $d_r+\deg y_r=\deg x_s+1<\deg tx_\ell$. If $\abs{V_{\leq s}+t V_{\leq s}}<q^{r}\abs{V_{\leq s}}$ then by induction $V_{\leq s}$ has strong arithmetic dimension of less than $r$, which is a contradiction as noted above. It follows that $\abs{V_{\leq s}+t V_{\leq s}}=q^r\abs{V_{\leq s}}$ and hence, since the $\dim V_{\leq s}+r$ many elements
\[\{ t^{n_i}y_i : 1\leq i\leq r, 0\leq n_i\leq d_i\}\]
span the $\bbf_q$-vector space $V_{\leq s}+t V_{\leq s}$, they are also linearly independent over $\bbf_q$. In particular, if 
\[\alpha_1t^{d_1}y_1+\cdots+\alpha_rt^{d_r}y_r\in V_{\leq s}\]
with $\alpha_i\in\bbf_q$ then we must have $\alpha_i=0$ for $1\leq i\leq r$. 

We now use the hypothesis $\abs{V+t V}\leq q^r\abs{V}$ to observe that the $\dim V+r+1$ elements 
\[\{ t^{n_i}y_i : 1\leq i\leq r, 0\leq n_i\leq d_i\}\cup \{ x_{s+1},\ldots,x_\ell\} \cup \{ tx_\ell\}\subset V+t V\]
are linearly dependent over $\bbf_q$. Since the degree of $tx_\ell$ is strictly larger than that of all elements of $V\cup t V_{\leq s}$ we in fact have that the set
\[\{ t^{n_i}y_i : 1\leq i\leq r, 0\leq n_i\leq d_i\}\cup \{ x_{s+1},\ldots,x_\ell\} \subset V+t V\]
is linearly dependent over $\bbf_q$. Therefore there must exist $\alpha_i\in\bbf_q$ for $1\leq i\leq r$, not identically zero, such that
\[z=\alpha_1t^{d_1}y_1+\cdots+\alpha_rt^{d_r}y_r\in V.\]
If $\alpha_r=0$ then $\deg z\leq d_{r-1}+\deg y_{r-1}< d_r+\deg y_r=\deg x_s+1$, and hence $z\in V_{\leq s}$, which contradicts the above. Hence we must have $\deg z=d_r+\deg y_r=\deg x_s+1$, and hence $z\in V_{\leq s+1}$. Let $1\leq i\leq r$ be such that $\alpha_i\neq 0$ and $d_i$ is minimal, and let $z=t^{d_i}y$, say. In particular we have that $t^jy\in V_{\leq s}$ for $0\leq j<d_i$ and $t^{d_i}y\in V_{\leq s+1}$. We claim that $V_{\leq s+1}$ has strong arithmetic dimension $r$, with a suitable decomposition provided by 
\begin{equation}
\label{vecspa}
\Pol{d_1}\cdot y_1\oplus\cdots\oplus \Pol{d_{i-1}}\cdot y_{i-1}\oplus \Pol{d_{i+1}}\cdot y_{i+1}\oplus\cdots\oplus \Pol{d_r}\cdot y_r\oplus\Pol{d_i+1}\cdot y.
\end{equation}

This contradicts the maximality of $s$ and completes the proof. The vector space \eqref{vecspa} is contained in $V_{\leq s+1}$, and since $\dim V_{\leq s+1}=\dim V_{\leq s}+1$, comparing dimensions shows that the vector spaces are equal, provided only that this sum is indeed direct. If the sum is not direct then we have $a_j\in \Pol{d_j}$, not identically zero, and $\beta\in\bbf_q$ such that
\[a_1y_1+\cdots+a_{i-1}y_{i-1}+a_{i+1}y_{i+1}+\cdots+a_ry_r+a_iy+\beta t^{d_i}y=0.\]
If $\beta=0$ then this contradicts the orthogonality of the original decomposition of $V_{\leq s}$, and so $\beta\neq 0$. Since the degree of the final summand is $d_i+\deg y=\deg z=d_r+\deg y_r$ which is strictly larger than the degree of all the other summands, the left hand side cannot be zero, which is a contradiction. Finally, the fact that this decomposition is a witness to $V_{\leq s+1}$ having strong arithmetic dimension $r$ follows from the fact that $d_i+1+\deg y=\deg z+1>d_r+\deg y_r$.
\end{proof}

\section{Transcendental dilates}\label{sec-dil}

We now apply the inverse theorem to a function field analogue of the following problem, first considered by Konyagin and {\L}aba \cite{KoLa:2006}. Let $A$ be a finite subset of $\bbr$ and $\xi$ be any transcendental element; must $\abs{A+\xi A}$ be large? Konyagin and {\L}aba proved that $\abs{A+\xi A}\gg (\log \abs{A})^{1-o(1)}\abs{A}$. Sanders \cite{Sa:2008a} later observed that, using an argument of Bourgain, such lower bounds can be obtained by combining simple modelling arguments with an inverse sumset result. Using such an argument with the sharpest known form of such an inverse result, Sanders \cite{Sa:2012b} improved this lower bound to
\[\abs{A+\xi A}\gg \exp(\Omega((\log \abs{A})^c))\abs{A}\]
for some absolute constant $c>0$. An example by Green, given in \cite{KoLa:2006}, shows that this is almost the best possible result. Namely, if one takes $A=\{ \sum_{i=1}^ma_i\xi^i : 1\leq a_i\leq n\}$ for suitable choices of $n$ and $m$ then one can show that 
\[\abs{A+\xi A}\ll \exp(O((\log \abs{A})^{1/2}))\abs{A}.\]

This problem over $\bbr$ has now been completely resolved by Conlon and Lim \cite{CL}, who prove that 
\[\abs{A+\xi A}\gg \exp(\Omega((\log \abs{A})^{1/2}))\abs{A}.\]
The proof of Conlon and Lim does not use an inverse sumset result, instead arguing geometrically and using compressions.

Over function fields the appropriate analogue of $\bbr$ is the field of Laurent series
\[\fk = \left\{ \sum_{n=-\infty}^k a_n t^n : a_n\in \mathbb{F}_p\textrm{ and }k\in \mathbb{Z}\right\},\]
which is the completion of the rational function field $\bbf_p(t)$. Since transcendence over $\bbf_p[t]$ is the obvious analogue of transcendence over $\bbz$, one might hope for similar lower bounds to the above to hold for $\abs{A+\xi A}$ when $A$ is any finite subset of $\fk$ and $\xi\in\fk$ is any element transcendental over $\bbf_p[t]$. 

A moment's thought shows that this is too ambitious; indeed, the analogue of the example outlined above already dashes our hopes. In particular, if $A=\{ \sum_{i=0}^na_i\xi^i : a_i\in \bbf_p\}$ then it is easy to show that $\abs{A+\xi A}\leq p\abs{A}$. Recalling that we take $\bbf_p$ to be fixed, this is essentially a constant upper bound, and hence no non-trivial lower bound can be given.

On examination of this example, however, some hope returns -- since such a set is contained in $\bbf_p[\xi]$ and $\xi$ is transcendental over $\bbf_p[t]$ we have $\abs{A+t A}\gg \abs{A}^2$. Thus one might hope that if $A\subset \fk$ does not grow when added to its dilation by some transcendental element, then this forces growth when added to its dilation by $t$.

The construction above is easily adapted to such a situation. Consider the set 
\[A=\left\{\sum_{i=1}^na_i\xi^i : a_i\in \bbf_p[t]\textrm{ and }\deg a_i <m\right\}.\]
It is easy to show that $\abs{A}=p^{nm}$ and $\abs{A+t A}=p^n\abs{A}$. Furthermore, $\abs{A+\xi A}=p^m\abs{A}$. It follows that if $\abs{A+t A}=K_1\abs{A}$ and $\abs{A+\xi A}=K_2\abs{A}$ then 
\[(\log K_1)(\log K_2)\asymp \log\abs{A}.\]

This should be compared to the case $A\subset \bbr$, when we study the single parameter $K$ given by $\abs{A+\xi A}=K\abs{A}$ and Green's construction gives a set $A$ with $(\log K)^2\approx \log \abs{A}$. Thus we see that in the analogous situation in $\bbf_p[t]$ there is a `splitting' of the parameter $K$ into two distinct parameters, and we may now ask for non-trivial lower bounds on the size of such parameters.

Combining Theorem~\ref{th-pfrff} with the argument of Sanders \cite{Sa:2008a} (which Sanders in turn attributes to Bourgain) we are able to prove an optimal such result. 

\begin{theorem}
Let $A\subset \fk$ be a finite set. If $\xi_1$ and $\xi_2$ are algebraically independent over $\bbf_p$  and $K_1,K_2\geq 2$ are such that
\[\abs{A+\xi_1 A}\leq K_1\abs{A}\textrm{ and }\abs{A+\xi_2 A}\leq K_2\abs{A}\]
then
\[(\log K_1)(\log K_2)\gg \log \abs{A}.\]
\end{theorem}
A similar conclusion may be obtained with $\bbf_q$ in place of $\bbf_p$, using Theorem~\ref{th-gen} instead of Theorem~\ref{th-pfrff}, but then an additional hypothesis such as $\abs{A+\lambda \xi_i A}\leq K_i^{O(1)}\abs{A}$ for all $\lambda\in \bbf_q$ and $i=1,2$ is required.

\begin{proof}
Without loss of generality, we may suppose that $K_2\geq K_1$ and $0\in A$, so that $A'=A\cup\xi_1 A\subseteq A+\xi_1 A$. Since $A'$ is finite there exists some $d\geq 1$ and $v_1,\ldots,v_d\in \fk$ linearly independent over $\bbf_p[\xi_2]$ with
\[A'\subseteq \left\{ x_1v_1+\cdots+x_dv_d : x_i\in \bbf_p[\xi_2]\right\}.\]
For some large integer $N$ we consider the map $f:\bbf_p[\xi_2]^d\cdot v \to \bbf_p[\xi_2]$ defined by
\[f(x_1v_1+\cdots+x_dv_d)=x_1+x_2\xi_2^N+\cdots+x_d\xi_2^{N(d-1)}.\]
Provided $N$ is large enough (depending on $\xi_1,\xi_2,A$) this is an injection on $A'$ and so if $A''=f(A')$ then $\abs{A''}=\abs{A'}\asymp \abs{A}$. Clearly
\[x+y=x'+y' \textrm{ implies }f(x)+f(y)=f(x')+f(y'),\]
and hence if $A''=f(A')$ then
\[\abs{A''+A''}\leq \abs{A'+A'}\leq \abs{A+A+\xi_1A+\xi_1A}\leq K_1^{O(1)}\abs{A''}.\]
Similarly
\[\abs{A''+\xi_2A''}\leq \abs{A'+\xi_2A'}\leq \abs{A+\xi_1 A+\xi_2 A+\xi_1\xi_2 A}\leq K_2^{O(1)}\abs{A''}.\]
By Theorem~\ref{th-pfrff} (applied with the transcendental $\xi_2$ playing the role of the transcendental $t$) there exists some $\bbf_p[\xi_2]$-arithmetic progression $P\subseteq \langle A''\rangle$ of rank $O(\log K_2)$ and size $\geq K_2^{-O(1)}\abs{A}$ such that $P\subseteq A''-A''+X$ for some set $X\subseteq P$ of size $\abs{X}\leq K_1^{O(1)}$. By the pigeonhole principle there exists some $x,y\in \bbf_p[\xi_2]$ and $m\geq 1$ such that
\[Q = x+y\cdot \bbf_p[\xi_2]_{\deg <m}\]
has size $\gg (K_2^{-O(1)}\abs{A})^{1/K'}$ for some $K'\ll \log K_2$ and $Q\subseteq A''-A''+X$. We define the inverse map $g:\bbf_p[\xi_2]\to \bbf_p[\xi_2]^d\cdot v$ by 
\[g(x_1+x_2\xi_2^N+\cdots+x_d\xi_2^{N(d-1)}+x_{d+1}\xi_2^{Nd})=x_1v_1+\cdots+x_dv_d,\]
where each $x_i$ (as a polynomial in $\bbf_p[\xi_2]$) has degree $<N$ for $1\leq i\leq d$. This is a well-defined map on $\{ x\in \bbf_p[\xi_2] : \deg x < N(d+1)\}$, which contains $Q$ for large enough $N$, and we can choose $N$ large enough such that 
\[g(Q)= g(x)+g(y)\cdot \bbf_p[\xi_2]_{\deg <m}\]
and
\[Q'=g(Q)\subseteq g(A''-A''+X)\subseteq A-A+\xi_1 A-\xi_1 A+g(X).\]
In particular, for any $l\geq 1$,
\[\Abs{Q'+\xi_1^2\cdot Q'+\cdots+\xi_1^{2(l-1)}\cdot Q'}\leq K_1^{O(l)}\Abs{(A-A)+\xi_1 (A-A)+\cdots+ \xi_1^{2l-1}(A-A)}.\]
Since $\xi_1$ is transcendental over $\bbf_p[\xi_2]$ and $Q'$ is a translate and dilate of a subset of $\bbf_p[\xi_2]$, however, the left hand side is at least $\abs{Q'}^{l}$. Furthermore, by Pl\"{u}nnecke-Ruzsa sumset estimates (as in \cite[Lemma 6.1]{Sa:2008a}) the right hand side is at most $K_1^{O(l)}\abs{A}$. It follows that 
\[K_2^{-O(l/K')}\abs{A}^{O(l/K')-1}\leq K_1^{O(l)}.\]
Choosing $l$ to be some large multiple of $K'\ll\log K_2$ we deduce that 
\[\log \abs{A} \ll (\log K_1)(\log K_2)\]
as required.
\end{proof}

We conclude by conjecturing the analogous result for subsets of the real numbers. 

\begin{conjecture}\label{conj}
Let $A\subset\bbr$ be a finite set. If $\xi\in \bbr$ is transcendental and $K_1,K_2\geq 2$ are such that
\[\abs{A+A}\leq K_1\abs{A}\textrm{ and }\abs{A+\xi A}\leq K_2\abs{A}\]
then
\[(\log K_1)(\log K_2)\gg \log \abs{A}.\]
\end{conjecture}
The Pl\"{u}nnecke-Ruzsa sumset inequalities imply that $\log K_1\ll \log K_2$, and hence this conjecture is stronger than the bound $\log K_2 \gg (\log \abs{A})^{1/2}$ of Conlon and Lim \cite{CL}.

The quasi-polynomial Bogolyubov-Ruzsa bounds of Sanders \cite{Sa:2012b} and a generalisation of the argument in \cite{Sa:2008a} (along similar lines to the previous proof) would give a lower bound of the shape $(\log K_1)(\log K_2)\gg (\log \abs{A})^c$ for some constant $c>0$. It is possible (although this is not immediately clear) that an adaptation of the argument of Conlon and Lim \cite{CL} would suffice to prove Conjecture~\ref{conj}. Similarly, it is likely that the polynomial Freiman-Ruzsa conjecture over the integers would imply Conjecture~\ref{conj}, although the fact that strong polynomial bounds would require not just generalised arithmetic progressions but more general convex progressions (see \cite{LR17}) mean that the argument above does not immediately generalise.

\subsection*{Acknowledgements}
The author is currently funded by a Royal Society University Research Fellowship, and when much of this work was done was supported by an EPSRC doctoral training grant. We would like to thank our PhD supervisor Trevor Wooley for the suggestion to explore additive combinatorics in function fields. We thank David Conlon, Akshat Mudgal, Will Sawin, and Terence Tao for helpful remarks and corrections to an early draft of this paper. Finally, we thank the anonymous referee for a careful reading of this paper and corrections.
\bibliographystyle{amsplain}

\begin{dajauthors}
\begin{authorinfo}[tb]
  Thomas F. Bloom\\
  Department of Mathematics\\
  University of Manchester\\
  Manchester, United Kingdom\\
  \texttt{thomas.bloom@manchester.ac.uk}
\end{authorinfo}
\end{dajauthors}

\end{document}